\documentclass[10pt, a4paper]{article}
\usepackage[english]{babel}
\usepackage{amsmath}
\usepackage{cmll}
\usepackage{amssymb}
\usepackage{amsthm}
\usepackage{bbm}
\usepackage{mathrsfs}
\usepackage{stmaryrd}
\usepackage[all]{xy}
\usepackage{bussproofs}

\newcommand{\sse}{\leftrightarrow}
\newcommand{\arr}{\longrightarrow}
\newcommand{\iso}{\simeq}
\newcommand{\imply}{\rightarrow}
\newcommand{\C}{\mathcal{C}}

\newcommand{\ee}{\mathcal{\epsilon}}

\newcommand{\PP}{P}
\newcommand{\pp}{(\C,P)}

\newdir{|>}{%
!/4.5pt/@{|}*:(1,-.2)@^{>}*:(1,+.2)@_{>}}

\mathcode`\:="603A     % : = \colon
\mathchardef\colon="303A  % :=

\mathcode`\<="4268     % < = \langle
\mathcode`\>="5269     % > = \rangle
\mathchardef\gt="313E  % arithmetic
\mathchardef\lt="313C  % strict order

\theoremstyle{definition}
\newtheorem{deff}{Definition}[section]
\newtheorem{prop}[deff]{Proposition}

\newtheorem{lem}[deff]{Lemma}

\newtheorem{Remark}[deff]{Remark}

\date{}
\begin{document}
\title{The axiom of choice, co-comprehension schema and redundancies in triposes}
\author{Fabio Pasquali}
\maketitle
%
%\begin{abstract}j
%\end{abstract}

\section*{Introduction}\label{sec0}

The notion of tripos has been introduced in \cite{tripos} and it turned out to be of central importance in category theory and categorical logic, see \cite{tripos,Jacobs, tripinret} and references therein. Every tripos is an instance of the more general concept of doctrine and, remaining at an informal level, a doctrine can seen as a pair $(\C,P)$ where $\C$ is a cartesian category endowed with a chosen logic $P$ (all formal definitions are postponed to section \ref{sec1}). Following this informal description of doctrines, we can say that, from a purely logical point of view, that a tripos is a doctrine expressing higher order intuitionistic logic with equality\footnote{ It is common to refer at equality in triposes as non-extesional. There are several notions of extensionality in categorical logic. One of these, which can be phrased using the structure of a tripos, is that for two terms $q$ and $p$ of type $\mathbb{P}(A)$ it is $q=p$ if and only if $\forall x:A\ x\in q \sse x\in p$. There are triposes for which this property is valid - it is enough to think at the internal logic of any elementary topos -  but that property fails to hold for generic triposes. In particular it fails in important triposes such as the realizability tripos and in some localic triposes \cite{Jacobs}. In this respect, the equality in triposes is non-extensional.}, i.e. the fragment
$$\mathbb{P}, \Pi, \Sigma, \imply, \wedge, \lor, \top, \bot, =$$
where we wrote $\mathbb{P}$ to indicate the fact that the logic is higher order, $\Pi$ do indicate universal quantification and $\Sigma$ to indicate existential quantification. \\\\
The structure of a tripos closely relates the category $\C$ with the logic $P$. To stress this correspondence we mention an important fact concerning the theory of triposes, the Tripos To Topos construction \cite{tripos,tripinret}. This is a construction which takes any tripos $(\C,P)$ and produces an elementary topos, usually denote by $\C[P]$, in such a way that the construction provides a left adjoint functor to the inclusion of the category of elementary toposes and logical functors is into the category of triposes and logical functors \cite{TTT}. Known examples of such toposes are the effective topos and any topos of sheaves over a locale \cite{tripinret}. Thus, saying that $\C$ is the base of a tripos, i.e that there is a pair $(\C,P)$ which is a tripos, is to say that $\C$ already possesses a sufficient structure to be freely completed to an elementary topos.\\\\ 
A fact of particular interest for this paper is that there are some redundancies in the notion of tripos. Using Prawitz 's second order encoding of first order predicates logic \cite{Pra}, it has been proved that a doctrine expressing the fragment $$\mathbb{P}, \Pi, \imply$$
is a tripos \cite{tripos, Jacobs}.\\\\
As one might aspect, the operation of implication $\imply$ is strongly connected to universal quantification $\Pi$. This connection can be outlined if the doctrine $(\C,P)$ has comprehension. In a doctrine $(\C,P)$ the property of having comprehension is the possibility to make correspond formulas of the logic $P$ to arrows of $\C$, in such a way that, if $\phi$ is the formula of $P$ to which corresponds the arrow $\lfloor\phi\rfloor$, we can think of the domain of $\lfloor\phi\rfloor$ as the abstract collection of those terms $t$ such that $\phi(t)$ holds. Doctrines with comprehension and triposes with comprehension share many interesting properties which we will not discuss here and we address interested readers to \cite{Giacobbe, Jacobs}.\\\\ A theorem in \cite{RM} states that a doctrine expressing the fragment $\{\}, \Pi, \wedge, \top$,  where we use $\{\}$ to denote that the doctrine has comprehension, has also the operation of implication. An obvious corollary is that a doctrine expressing the fragment
$$\mathbb{P}, \{\}, \Pi, \wedge, \top$$
is a tripos with comprehension.\\\\
In this paper we study the notion of co-comprehension (the dual notion of comprehension) and the axiom of choice in the context of doctrines. In particular we study how these two properties contribute to definability, i.e. in reconstructing a certain structure from a (apparently) poorer one.\\\\
Co-comprehension is the dual notion of comprehension, i.e the possibility to associate to every formula $\phi$ of $P$ a monic $\lceil\phi\rceil$ of $\C$ in such a way that the domain of $\lceil\phi\rceil$ abstractly represents the collection of those terms $t$ such that $\phi(t)$ does not hold. Of course in classical logic both the notions of comprehension and co-comprehension are derivable from the other, but they are independent if we work in an intuitionistic framework.\\\\
We formulate the axiom of choice (AC) in a doctrine $(\C,P)$ by requiring that every existentially quantified formulas of $P$ has a chosen witness (as in Hilbert's epsilon calculus).\\\\
We prove that if a cartesian category $\C$ is the base of a doctrine $\pp$ expressing the fragment 
$$\text{AC}, \{\}^o, \mathbb{P}, \wedge,\bot, =$$
(where we denote by $\{\}^o$ the property of having co-comprehension) then $\C$ is also the base of a tripos $(\C,P')$. In other words $\C$ already has the sufficient structure to freely build a topos. The way in which $P'$ is built is instrumental to prove our second result, i.e. that any doctrine expressing the fragment 
$$\text{AC}, \{\}^o, \mathbb{P}, \imply, =$$
is already a tripos.\\\\
The paper is organized as follows. In section \ref{sec1} we recall definitions and known facts concerning doctrines and triposes. We recall also the notion of comprehension and full comprehension. Even if we did not stress the difference between comprehension and full comprehension in the previous part of the introduction, this is very important in the paper (and for the theory of doctrines in general). The role of fullness of comprehension in relation with the operation of implication is discussed in section \ref{sec4}, where we prove that every doctrine with full comprehension and right adjoints which satisfy (a form of) Beck-Chevalley condition is implicational. In the same section we introduce also the notion of co-comprehension and full co-comprehension and we study its link with the operation of negation. In particular we study the link between fullness of co-comprehension and classical logic. In section \ref{sec2} we give the definition of doctrine validating AC and we develop a piece of theory of these doctrines. In section \ref{sec5} we study those doctrines that have equality predicates, full co-comprehension and which satisfy the axiom of choice. We prove that if AC well behaves with respect to equality and co-comprehension in a higher order doctrine $\pp$, then we can build on $\C$ a tripos $(\C,P')$. We prove also that, in the case $\pp$ has the operation of implication, is already a tripos.
\section{Preliminaries on doctrines and triposes}\label{sec1}
In this section we recall several definitions concerning doctrines. We refer the reader to \cite{Jacobs,PittsCL} for a more detailed account.
\subsection*{Doctrines}
Denote by \textbf{Pos} the category of partially ordered sets and monotone functions and by \textbf{MSL} the subcategory of \textbf{Pos} on meet-semilattices and homomorphisms between them, where a meet-semilattice is a poset with binary meets.
\\\\
A doctrine is a pair $(\C,P)$ where $\C$ is a category with finite products and $P$ is a functor $$P:\C^{op}\arr\textbf{Pos}$$
$(\C,P)$ is said \textbf{primary} if $P$ factors through \textbf{MSL}.\\\\
%\footnote{This definition of primary doctrine may differ from the usage of some authors who require that all posets of the form $P(A)$ have finite meets.}.\\\\
In lattices of the form $P(A)$ we shall denote binary meets by $\wedge$ and the top element, if it exists, by $\top_A$. The category $\C$ is called base of the doctrine. For every arrow $f$ of the base, we shall write $f^*$ instead of $P(f)$.\\\\
Suppose $\pp$ is a doctrine. Suppose $u:X\arr A$ is an arrow of the base. Suppose that there are functors $$\Pi_u,\Sigma_u:\PP(X)\arr\PP(A)$$ such that for every $\alpha$ in $\PP(X)$ and every $\beta$ in $\PP(A)$ it is $$\alpha\le u^*\Sigma_u\alpha\ \ \ \ \ \ \text{and}\ \ \ \ \ \ \Sigma_uu^*\beta\le\beta$$
$$u^*\Pi_u\alpha\le\alpha\ \ \ \ \ \ \text{and}\ \ \ \ \ \ \ \beta\le\Pi_uu^*\beta$$
$\Sigma_u$ is said to be left adjoint to $u^*$, while $\Pi_u$ is said to be right adjoint to $u^*$. We denote this by $\Sigma_u\dashv u^*$ and $u^*\dashv\Pi_u$.\\\\
A doctrine $\pp$ is \textbf{elementary} if it is primary and for every $A$ in $\C$ there is an element $\delta_A$ in $P(A\times A)$ such that for every arrow $f:A\arr X$ in $\C$ and every $\psi$ in $P(X\times A)$ the assignment $$\psi\mapsto <\pi_1,\pi_2>^*\psi\wedge<\pi_2,\pi_3>^*\delta_A$$ gives rise to a functor $$\Sigma_{(id_X\times \Delta_A)}:P(X\times A)\arr P(X\times A\times A)$$ which is left adjoint to $(id_X\times \Delta_A)^*$, where $\pi_i$, with $i=1,2,3$, are the projections from $X\times A\times A$ to each of the factors and $\Delta_A = <id_A,id_A>:A\arr A\times A$ is the diagonal arrow of $A$.\\\\
The element $\delta_A$ in $P(A\times A)$ is often called \textbf{equality predicate} over $A$. Equality predicates are \textbf{substitutive}, i.e for every $A$ in $\C$ and every $\psi$ in $P(A)$ $$\pi_1^*\psi\wedge\delta_A = \pi_2^*\psi\wedge\delta_A$$
where $\pi_1$ and $\pi_2$ are the first and second projection from $A\times A$.\\\\
Let $\mathcal{A}$ be a class of arrows of $\C$ which is stable under pullbacks, i.e. if $f:A\arr B$ is in $\mathcal{A}$ and $p:X\arr B$ is an arrow in $\C$, the pullback of $f$ along $p$, if it exists, is in $\mathcal{A}$. Suppose now that for every arrow $f:A\arr B$ in $\mathcal{A}$ the functor $f^*$ has a left adjoint $\Sigma_f$. We say that left adjoints satisfy the \textbf{Beck-Chevalley condition} with respect to $\mathcal{A}$ if for every pullback
\[
\xymatrix{
Y\ar[r]^-{g}\ar[d]_-{k}&X\ar[d]^-{h}\\
A\ar[r]_-{f}&B
}
\] 
with $f$ (and therefore $g$) in $\mathcal{A}$, it holds that $$h^*\Sigma_f \gamma\iso \Sigma_gk^*\gamma$$for any $\gamma$ in $\PP(A)$. We say that left adjoints verify the \textbf{restricted Beck-Chevalley condition} with respect to $\mathcal{A}$ if for any $\xi$ in $\PP(B)$ it holds $$h^*\Sigma_f f^*\xi\iso \Sigma_gk^*f^*\xi$$ Definitions for right adjoints $\Pi$ are analogous.\\\\
Suppose $\mathcal{A}$ is a pullback stable class of morphisms of $\C$. Let $\pp$ be a doctrine.
$\pp$ is a \textbf{(restricted) $\Sigma(\mathcal{A})$-doctrine} if for every arrows $f$ in $\mathcal{A}$ the functor $f^*$ has a left adjoint $\Sigma_{f}$ satisfying the (restricted) Beck-Chevalley condition with respect to $\mathcal{A}$.\\\\
Recall that in every category with finite products the class $Prj$ of product projections, i.e. the class of arrows of the form $\pi_A:X\arr A$ where $A$ appears as a factor in $X$, is stable under pullbacks. We shall write $\Sigma$-doctrine instead of $\Sigma(Prj)$-doctrine.\\\\ $\Pi(\mathcal{A})$-doctrines, restricted $\Pi(\mathcal{A})$-doctrines and $\Pi$-doctrines are defined analogously.\\\\
A primary $\Sigma(\mathcal{A})$-doctrine in which for every $f:X\arr A$ in $\mathcal{A}$, for every $\alpha$ in $P(X)$ and every $\beta$ in $P(A)$ it is $$\Sigma_{f}(\alpha\wedge f^*\beta) = \beta\wedge\Sigma_{f}\alpha$$
is said to satisfy \textbf{Frobenius Reciprocity}. $\Sigma$-doctrines which satisfy Frobenius Reciprocity are called \textbf{existential}.\\\\
Every existential elementary doctrine $(\C,P)$ has the property that for every arrow $f:A\arr B$ in $\C$ the functor $f^*$ has a left adjoint which is computed by the following assignment
$$\Sigma_f\alpha=\Sigma_{\pi_B}[(f\times id_B)^*\delta_B\wedge\pi_A^*\alpha]$$
where $\pi_B$ and $\pi_A$ are projections from $A\times B$ to each of the factors \cite{Jacobs, tripinret}.
\subsection*{Comprehension}
A doctrine $(\C,P)$ is said to have \textbf{comprehension} if for every $A$ in $\C$ the poset $P(A)$ has a top element $\top_A$ and for every $\alpha$ in $P(A)$ there is a morphism $$\lfloor\alpha\rfloor:\{\alpha\}\arr A$$ with $\lfloor\alpha\rfloor^*\alpha=\top_{\{\alpha\}}$ such that for every arrow $f:Y\arr A$ with $f^*\alpha=\top_Y$ there is a unique arrow $k:Y\arr \{\alpha\}$ with $\lfloor\alpha\rfloor k= f$.\\\\
It is an easy check to verify that arrows of the form $\lfloor\alpha\rfloor$ are monic. For every $A$ in $\C$ we denote by $\textbf{sub}_\C(A)$ the poset of subobjects of $A$. The assignment $\alpha\mapsto \lfloor\alpha \rfloor$ is a faithful functor from the poset $\PP(A)$ to the poset $\textbf{sub}_\C(A)$\footnote{As customary, we freely confuse a subobject with any of its representative.}. If it is also full we say that the doctrine $(\C,P)$ has \textbf{full comprehension}.\\\\
A useful property of full comprehension is that for every $\alpha,\beta$ in $P(A)$ it holds that $\alpha\le\beta$ if and only if $\top_{\{\alpha\}}=\lfloor\alpha\rfloor^*\beta$.\\\\
Suppose $\pp$ has comprehension. Denote by $\C_P$ the class of arrows of $\C$ of the form $\lfloor\alpha\rfloor$. Then $\C_P$ is stable under pullbacks, in fact for every $A$ in $\C$, every $\alpha$ in $P(A)$ and every $f:X\arr A$ the following square
\[
\xymatrix{
\{f^*\alpha\}\ar[d]_-{\lfloor f^*\alpha\rfloor}\ar[r]^-{q}&\{\alpha\}\ar[d]^-{\lfloor\alpha\rfloor}\\
X\ar[r]_-{f}&A
}
\] 
where $q$ is the arrow induced by the universal property of $\lfloor\alpha\rfloor$,  is a pullback.
\subsection*{Weak power objects}
Suppose $\pp$ is a doctrine. We say that $\pp$ has \textbf{weak power objects} if for every $A$ in $\C$ there is an object $\mathbb{P}(A)$ in $\C$ and an element $\in_A$ in $\PP(A\times \mathbb{P}(A))$ such that for every $Y$ in $\C$ and every $\phi$ in $\PP(A\times Y)$ there is an arrow $$\chi_\phi: Y\arr\mathbb{P}(A)$$ with $(id_A\times \chi_\phi)^*\in_A=\phi$.\\\\\
The term weak comes from the fact that arrows of the form $\chi_\phi$ need not to be unique. A doctrine is said \textbf{higher order} if it has weak power objects.
%It can be shown \cite{tripinret} that for every $f:A\arr X$ in $\C$ it is
%$$\Sigma_f\top_A = \Sigma_{\pi_X}((f\times id_X)^*\delta_X$$
\subsection*{Triposes}
Triposes are the appropriate doctrines to express higher order predicates logic \cite{Jacobs}. 
The following definition of tripos is deduced from \cite{tripinret}.\\\\Denote by \textbf{Hyt} the category of Heyting algebras and homomorphisms between them. A doctrine $\pp$ is said  \textbf{propositional} if $P$ factors through \textbf{Hyt}.\\\\$\pp$ is a \textbf{tripos} if 
\begin{itemize}
\item[i)] $\pp$ is propositional
\item[ii)] $\pp$ is a $\Sigma$-doctrine and a $\Pi$-doctrine
\item[iii)] for each diagonal $\Delta_X:X\arr X\times X$ there is $\delta_X$ in $\PP(X\times X)$ such that for every $\alpha$ in $\PP(X\times X)$
$$\top_X\le\Delta_X^*\alpha\ \ \ \ \ \ \text{if and only if}\ \ \ \ \ \ \delta_X\le\alpha$$
\item[iv)] $\pp$ is higher order.
\end{itemize}
Triposes are elementary doctrines, where the elementary structure is provided by the point $iii$ of the definition.\\\\
An important characterization theorem of the theory of triposes is the following.
\begin{prop}\label{redundus} A doctrine $\pp$ is a tripos if and only if 
\begin{itemize}
\item[i)] $\pp$ is a $\Pi$-doctrine
\item[ii)] for every $A$ in $\C$ there is an operation $\imply:P(A)\times P(A)\arr P(A)$ such that for every $f:X\arr A$ and every $\alpha,\beta$ in $P(A)$ it is $f^*(\alpha\imply\beta)=f^*\alpha\imply f^*\beta$
\item[iii)] for every projection $\pi_A:X\times A\arr A$ in $\C$ it is $$\Pi_{\pi_A}(\pi_A^*\alpha\imply \beta) = \alpha\imply\Pi_{\pi_A} \beta $$
for any $\alpha$ in $P(A)$ and $\beta$ in $P(X\times A)$
\item[iv)] for every $\gamma,\phi$ and $\psi$ in $P(A)$ it holds
\begin{itemize}
\item[a)] $\phi\le\psi\imply\phi$
\item[b)]$\gamma\imply(\phi\imply\psi)\le(\gamma\imply\phi)\imply(\gamma\imply\psi)$
\item[c)]if $\gamma\le\phi\imply\psi$ and $\gamma\le\phi$, then $\gamma\le\psi$
\item[d)]if $\phi\le\psi$, then $\gamma\le\phi\imply\psi$
\end{itemize}
\item[v)] $\pp$ is higher order.
\end{itemize}
\end{prop} 
This is theorem 1.4 of \cite{tripos}, whose proof basically relies on the second order encoding \`a la Prawitz \cite{Pra}. A similar proof is also in \cite{Jacobs}.\\\\
A doctrine which satisfies point ii), iii) and iv) of the previous definition is said \textbf{implicational}.\\\\
An easy property which a primary implicational doctrine has is that if it is a $\Sigma(\mathcal{A})$-doctrine, for some pullback stable class $\mathcal{A}$ of arrows of its base, then it satisfies Frobenius reciprocity \cite{Jacobs}. Therefore primary implicational $\Sigma$-doctrines are automatically existential.
\section{(co-)Comprehension and triposes}\label{sec4}
In this section we derive many useful properties that doctrines have under the assumption of having comprehension. Then we introduce the notion of co-comprehension, the dual notion of comprehension, and we study its connection with the operation of negation that a doctrine may have. A special focus will be given to triposes.\\\\
Recall that in a doctrine $\pp$ with comprehension, the collection $\C_P$ is the class of arrows of $\C$ of the form $\lfloor\alpha\rfloor:\{\alpha\}\arr A$. We already remarked that $\C_P$ is closed under pullbacks. 
\begin{prop}\label{nonloso}
If $(\C,P)$ is a primary doctrine with full comprehension, if for every arrow $\lfloor\alpha\rfloor$ of $\C_P$, the functor $\lfloor\alpha\rfloor^*$ has a left adjoint satisfying Frobenius reciprocity, then $\pp$ is a restricted $\Sigma(\C_P)$-doctrine.
\end{prop}
\begin{proof}Note that for every $A$ in $\C$ and every $\alpha$ in $\PP(A)$ it is $$\exists_{\lfloor\alpha\rfloor}\top_{\{\alpha\}} =  \alpha$$
In fact $\lfloor\alpha\rfloor$ trivially factors through $\lfloor\exists_{\lfloor\alpha\rfloor}\top_{\{\alpha\}}\rfloor$ and by fullness of comprehension we get $\alpha \le \exists_{\lfloor\alpha\rfloor}\top_{\{\alpha\}}$. The other inequality is standard.\\\\
Given the pullback
\[
\xymatrix{
\{f^*\alpha\}\ar[d]_-{\lfloor f^*\alpha\rfloor}\ar[r]^-{q}&\{\alpha\}\ar[d]^-{\lfloor\alpha\rfloor}\\
X\ar[r]_-{f}&A
}
\] 
by Frobenius Reciprocity we have \begin{equation}\notag
\begin{split}
f^*\Sigma_{\lfloor\alpha\rfloor}\lfloor\alpha\rfloor^*\xi & = f^*(\xi\wedge\Sigma_{\lfloor\alpha\rfloor}\top_{\{\alpha\}} )\\
& = f^*\xi\wedge f^*\alpha\\
& = f^*\xi\wedge \Sigma_{\lfloor f^*\alpha\rfloor}\top_{\{f^*\alpha\}}\\
& = \Sigma_{\lfloor f^*\alpha\rfloor}\lfloor f^*\alpha\rfloor f^*\xi\\
& = \Sigma_{\lfloor f^*\alpha\rfloor}q^*\lfloor \alpha\rfloor^* \xi
\end{split}
\end{equation}
hence the claim.\end{proof}
An immediate corollary of the previous proposition is that every tripos $\pp$ with full-comprehension is a restricted $\Sigma(\C_P)$-doctrine. In fact triposes are implicational existential doctrines, then they left left adjoints satisfy Frobenius reciprocity.\\\\
The following lemma relates full comprehension with the implicational structure of a doctrine.
\begin{lem}\label{bingo} Suppose $(\C,P)$ is a $\Pi$-doctrine with full comprehension, if $(\C,P)$ is a restricted $\Pi(\C_P)$-doctrine then it is implicational.
\end{lem}
\begin{proof}For every $A$ in $\C$ and every $\phi,\psi$ in $P(A)$ define$$\phi\imply\psi = \Pi_{\lfloor\phi\rfloor}\lfloor\phi\rfloor^*\psi$$We need to verify points ii), iii) and iv) of definition \ref{redundus}.\\\\
ii) Suppose that $f:X\arr Y$ is an arrow in $\C$. By the restricted Beck-Chevalley condition on the class $\C_P$ we have $$f^*\Pi_{\lfloor\phi\rfloor}\lfloor\phi\rfloor^*\psi = \Pi_{\lfloor f^*\phi\rfloor}\lfloor f^*\phi\rfloor^*f^*\psi$$
which proves that $f^*(\alpha\imply\beta) = f^*\alpha\imply f^*\beta$.\\\\
iii) Suppose that $\pi_A:X\times A\arr A$ is a projection in $\C$. Suppose $\alpha$ is in $P(A)$ and $\beta$ in $P(X\times A)$. Consider the pullback
\[
\xymatrix{
\{\pi_A^*\alpha\}\ar[r]^{p_A}\ar[d]_-{\lfloor\pi_A^*\alpha\rfloor}&\{\alpha\}\ar[d]^-{\lfloor\alpha\rfloor}\\
X\times A\ar[r]_-{\pi_A}&A
}
\] 
Since $(\C,P)$ is a $\Pi$-doctrine the Beck-Chevalley condition holds for $\pi_A$ and $p_A$ (which is a projection, since the square is a pullback). Then it is
\begin{equation}\notag
\begin{split}
\Pi_{\pi_A}(\pi_A^*\alpha\imply \beta) & = \Pi_{\pi_A}\Pi_{\lfloor\pi_A^*\alpha\rfloor}\lfloor\pi_A^*\alpha\rfloor^*\beta\\
& =\Pi_{\lfloor\alpha\rfloor}\Pi_{p_A}\lfloor\pi_A^*\alpha\rfloor^*\beta\\
& =\Pi_{\lfloor\alpha\rfloor}\lfloor\alpha\rfloor^*\Pi_{\pi_A}\beta\\
& =\alpha\imply\Pi_{\pi_A}\beta
\end{split}
\end{equation}
iv) Suppose $\phi, \psi, \gamma$ are in $P(A)$.\\\\
iv-a) This is an immediate consequence of the fact that $\Pi_{\lfloor\psi\rfloor}$ is a right adjoint.\\\\
iv-b) Note that 
\begin{equation}\notag
\begin{split}
(\gamma\imply\phi)\imply(\gamma\imply\psi) & = \Pi_{\lfloor\gamma\imply\phi\rfloor}\lfloor\gamma\imply\phi\rfloor^*(\gamma\imply\psi)\\
& = \Pi_{\lfloor\gamma\imply\phi\rfloor}(\lfloor\gamma\imply\phi\rfloor^*\gamma\imply\lfloor\gamma\imply\phi\rfloor^*\psi)\\
& = \gamma\imply\Pi_{\lfloor\gamma\imply\phi\rfloor}\lfloor\gamma\imply\phi\rfloor^*\psi\\
& = \gamma\imply((\gamma\imply\phi)\imply\psi)
\end{split}
\end{equation}
The arrow  $\lfloor\top_A\rfloor$ is the identity on $A$, hence $\top_A\imply \phi = \Pi_{\lfloor\top_A\rfloor}\lfloor\top_A\rfloor^*\phi = \phi$. Recalling that $\lfloor \gamma\rfloor^*\gamma = \top_A$, we get the following associativity law $$\lfloor\gamma\rfloor^*(\gamma\imply(\phi\imply\psi))=\lfloor \gamma\rfloor^*(\phi\imply\psi) = \lfloor\gamma\rfloor^*((\gamma\imply\phi)\imply\psi)$$  
whence
\begin{equation}\notag
\begin{split}
\gamma\imply(\phi\imply\psi)& \le \Pi_{\lfloor\gamma\rfloor}\lfloor\gamma\rfloor^*((\gamma\imply\phi)\imply\psi)\\
&= \gamma\imply((\gamma\imply\phi)\imply\psi)
 \end{split}
\end{equation}
iv-c) Assume $\gamma\le\phi\imply\psi$ and $\gamma\le\phi$. Since $\lfloor\gamma\rfloor^*\gamma=\top_A$, we have $\lfloor\gamma\rfloor^*\phi\le\lfloor\gamma\rfloor^*\psi$ and hence $\top_A\le\lfloor\gamma\rfloor^*\psi$. By fullness of comprehension $\gamma\le\psi$.\\\\
iv-d) Suppose $\phi\le\psi$. It is $\gamma\le\Pi_{\lfloor\phi\rfloor}\lfloor\phi\rfloor^*\psi$ if and only if $\lfloor\phi\rfloor^*\gamma\le\lfloor\phi\rfloor^*\psi$. Since $\lfloor\phi\rfloor^*\phi = \top_A$, the assumption that $\phi\le\psi$ implies $\lfloor\phi\rfloor^*\psi=\top_A$, from which the claim.
\end{proof}
\begin{Remark}We did not find any reference for lemma \ref{bingo}. Even if the author is aware that similar statements are known in the case in which the doctrine $\pp$ is primary. In this respect the following statement is more general as it does not require a meet operation on posets. If the doctrine $(\C,P)$ is primary the statement is proved by showing that  $\imply$ is a Heyting implication, i.e. $\alpha \wedge \phi \le \psi$ if and only if $\alpha\le \phi\imply\psi$. This proof can be found in \cite{RM}.\end{Remark}
Fullness of comprehension might appear to play a little role in the proof, but in a $\Pi$-doctrine with comprehension the assignment $$(\alpha,\beta) \mapsto \Pi_{\lfloor\alpha\rfloor}\lfloor\alpha\rfloor^*\beta$$provides the implicational structure if and only if comprehension is full. One direction is in the previous theorem. For the converse we have that $\lfloor\alpha\rfloor^*\alpha\le\lfloor\alpha\rfloor^*\beta$ if and only if $\alpha\le\alpha\imply\beta$. Then $\alpha\le\beta$ by iv-c.\\\\
Co-comprehension is the dual notion of comprehension. We say that a doctrine $\pp$ has \textbf{co-comprehension} if for every $A$ in $\C$ the poset $P(A)$ has a bottom element $\bot_A$ and for every $\alpha$ in $\PP(A)$ there is an arrow $$\lceil\alpha\rceil:\{\alpha\}^o\arr A$$ with $\lceil\alpha\rceil^*\alpha \iso \bot_X$ such that for every $f:Y\arr A$ with $f^*\alpha\iso\bot_Y$ there is a unique arrow $k:Y\arr$ with $\lceil\alpha\rceil k = f$.\\\\
Arrows of the form $\lceil\alpha\rceil$ are monic. Thus if $\pp$ has co-comprehension, there is a faithful contravariant functor from $\PP(A)$ to $\textbf{sub}_\C(A)$ which is given by the assignment $\alpha\mapsto\lceil\alpha\rceil$. We say that $\pp$ has full co-comprehension if the previous assignment is also full. If this is the case, it is not hard to prove that $\alpha\le\beta$ if and only if $\lceil\beta\rceil^*\alpha=\bot_{\{\alpha\}^o}$.\\\\  
Let $\pp$ be a doctrine with co-comprehension. Suppose $\alpha$ is in $\PP(A)$ and consider any arrow $f:X\arr A$. The following is a pullback
\[
\xymatrix{
\{f^*\alpha\}^o\ar[r]^-{q}\ar[d]_-{\lceil f^*\alpha\rceil}&\{\alpha\}^o\ar[d]^-{\lceil\alpha\rceil}\\
X\ar[r]_-{f}&A
}
\]
where $q$ is the arrow induced by the universal property of $\lceil\alpha\rceil$.
Hence the class $\C_P^o$ of arrows of $\C$ of the form $\lceil\alpha\rceil$ is closed under pullbacks.\\\\
The notion of co-comprehension is related to the notion of comprehension in the case in which the doctrine as a operation of negation.\\\\
Recall that a primary doctrine \textbf{has negation} if for every $A$ in $\C$ the poset $P(A)$ has a bottom element and there is an operation $\neg:P(A)\arr P(A)$ such that for every $\alpha,\beta$ in $P(A)$, it is $$\alpha\le \neg\beta\ \ \ \ \ \text{if and only if}\ \ \ \ \  \alpha\wedge\beta =\bot_A$$
If it holds also that $\alpha = \neg\neg \alpha$ the doctrine is said \textbf{classical}
\begin{prop} Suppose $\pp$ is a doctrine with comprehension and negation. The following hold.
\begin{itemize}
\item[i)] $\pp$ has co-comprehension.
\item[ii)] if $\pp$ is a restricted $\Sigma(\C_P)$-doctrine, it is also a restricted $\Sigma(\C_P^o)$-doctrine
\item[iii)] if $\pp$ has full comprehension, then $\pp$ has full co-comprehension if and only if $\pp$ is classical.
\end{itemize}
\end{prop}
\begin{proof} i) Suppose $\alpha$ is in $P(A)$. Define $$\lceil\alpha\rceil:\{\alpha\}^o\arr A = \lfloor\neg\alpha\rfloor:\{\neg\alpha\}\arr A$$
It is $\top_A= \lfloor\neg\alpha\rfloor^*(\neg\alpha)$ then $\bot_A = \neg \top_A= \neg\lfloor\neg\alpha\rfloor^*(\neg\alpha)=\lfloor\neg\alpha\rfloor^*(\neg\neg\alpha)$. Since $\alpha\le\neg\neg\alpha$ we have $\bot_A= \lfloor\neg\alpha\rfloor^*\alpha$.\\\\
If $f:Y\arr A$ is such that $f^*\alpha=\bot_Y$, then $\top_Y=\neg f^*\alpha = f^*(\neg \alpha)$. By the universal property of comprehension, there is a unique $k:Y\arr \{\alpha\}^o$ with $f=\lfloor\neg\alpha\rfloor k = \lceil\alpha\rceil k$.\\\\
ii) Since arrows in $\C_P^o$ are of the form $\lfloor\neg\alpha\rfloor$, $\C_P^o$ constitutes a pullback stable subclass of $\C_P$. Then the restricted Beck-Chevally condition with respect to $\C_P^o$ is inherited from the restricted Beck-Chevalley condition with respect to $\C_P$.\\\\
iii) Suppose $\pp$ has full comprehension. If $\lfloor\neg\alpha\rfloor^*\beta\le\lfloor\neg\alpha\rfloor^*\alpha$ it is also $\lfloor\neg\alpha\rfloor^*\neg\alpha\le\lfloor\neg\alpha\rfloor^*\neg\beta$. By fullness of comprehension we have $\neg\alpha\le\neg\beta$. Then $\beta\le\alpha$ if and only if $\pp$ is classical. 
\end{proof}
We proved that in a doctrine with full-comprehension, a classical negation implies the existence of full co-comprehension. If we think at co-comprehension as a way to deal with abstract complementation, we may guess that the presence of both full comprehension and full co-comprehension implies the existence of a classical negation. The answer is negative. Consider the doctrine $(\textbf{Top}, \textbf{op})$, where $\textbf{Top}$ is the category of topological spaces and continuous functions and \textbf{op} is the functor which maps every topological space into the collection of its open sets and every continuous function to the inverse image functor. For every space $X$ and every open set $U$ of $X$ the inclusions $$U\hookrightarrow X\ \ \ \ \ \text{and}\ \ \ \ \ U^c\hookrightarrow X$$
where $U$ and its complement $U^c$ are endowed with the subspace topology, provide the full comprehension and the full co-comprehension structure for $(\textbf{Top}, \textbf{op})$, which has not the operation of negation \cite{cofree}.
\section{The axiom of choice in doctrines}\label{sec2}
In this section we introduce the axiom choice in the context of doctrines. The idea is to mimic what happen in the doctrine $(\textbf{Sets}, \mathcal{P})$ where the base is the category of sets and functions and $\mathcal{P}$ is the contravariant powersets functor. The axiom of choice can be phrased in the following way: for every binary relation $\psi$ in $\mathcal{P}(\Gamma\times A)$, where the set $A$ is not empty, there is function $\ee_\psi:\Gamma\arr A$ such that for all element $x$ in $\Gamma$
$$\exists a\in A\ (x,a)\in \phi \sse (x, \ee_\phi(x))\in \phi$$ 
Clearly the fact that $A$ is not empty is not empty is a necessary condition for the existence of the function $\ee_\psi$. Then, to deal with the axiom of choice, in the way we have formulated it, we need a notion of emptiness, at the general level of doctrines. The generalization of the empty set in the context of a generic doctrine is provided by the notion of stable initial object \cite{PittsCL}, i.e an object $0$ of $\C$ which is initial and such that $X\times 0\iso0$ for any $X$.
\begin{deff}\label{prima}
A doctrine $\pp$ satisfies the axiom of choice (AC) if for every object $A$ in $\C$ which is not a stable initial object, for every object $\Gamma$ in $\C$ the functor $\pi_\Gamma^*$ has a left adjoint $\Sigma_{\pi_\Gamma}$ and for every $\psi$ in $\PP(\Gamma\times A)$ there exists an arrow $\ee_\psi:\Gamma\arr A$ such that $$\Sigma_{\pi_\Gamma}\psi = <id_\Gamma,\ee_\psi>^*\psi$$
\end{deff}
We now prove that, under some mild conditions, every doctrine validating the axiom of choice (AC) is a $\Sigma$-doctrine.
\begin{lem}\label{zero} If a $\pp$ validates $AC$, if $\C$ has a stable initial object and if posets $P(A)$ are not empty, then all arrows of the form $k:X\arr 0$ are isomorphisms.
\end{lem}
\begin{proof} If $k$ is not iso, $X$ is not a stable initial object, and since $P(1\times X)$ contains an element $x$, by AC we have an arrow $\ee_x:1\arr X$. Then $k\ee_x:1\arr 0$ makes $\C$ degenerate. A contradiction.
\end{proof}
\begin{lem}\label{BC} Suppose $\pp$ is a doctrine satisfying AC. For every projection $\pi_\Gamma:\Gamma\times A\arr \Gamma$ where $A$ is not a stable initial object, for every $\psi\in \PP(\Gamma\times A)$ and every $h:\Gamma\arr A$, it is $$<id_\Gamma, h>^*\psi\le<id_\Gamma,\ee_\psi>^*\psi$$
\end{lem}
\begin{proof} Apply $<id_\Gamma, h>^*$ to both sides of $\psi\le\pi_\Gamma^*\Sigma_{\pi_\Gamma}\psi=\pi_\Gamma^* <id_\Gamma,\ee_\psi>^*\psi$.
\end{proof}
\begin{prop}\label{nonne0} Suppose $\pp$ is a doctrine satisfying AC, such that for every $A$ in $\C$ the poset $P(A)$ has a bottom element. Suppose also that if $0$ is a stable initial object, then $P(0)$ is a singleton. $\pp$ is a $\Sigma$-doctrine.
\end{prop}
\begin{proof}
We need to prove that for every projection $\pi_\Gamma:\Gamma\times A\arr\Gamma$ the left adjoint $\Sigma_{\pi_\Gamma}$ satisfying the Beck-Chevalley condition.\\\\
Suppose $A$ is not a stable initial object. By AC it is $$\Sigma_{\pi_\Gamma}\psi = <id_\Gamma,\ee_\psi>^*\psi$$ Consider an arrow $f:\Delta\arr\Gamma$ and the composition $\ee_\psi f:\Delta\arr A$. By \ref{BC} it is $$<id_\Delta,\ee_\psi f>^*(f\times id_A)^*\psi\le <id_\Delta,\ee_{P(f\times id_A)(\psi)}>^*(f\times id_A)^*\psi$$
since $$(f\times id_A)<id_\Delta,\ee_\psi f>=<f,\ee_\psi f>=<id_\Gamma,\ee_\psi>f$$ the previous inequality can be rewritten as $$f^*\Sigma_{\pi_\Gamma}\psi\le\Sigma_{\pi_\Delta}(f\times id_A)^*\psi$$
which proves the claim, as the other inequality is standard.\\\\
Suppose $A\iso0$ is a stable initial object. Since $0$ is stable we can confuse the projection $\pi_\Gamma$  with the unique arrow $!_\Gamma:0\arr\Gamma$. Since $P(0)$ is a singleton, the assignment $\bot_0\mapsto\bot_\Gamma$ provides a left adjoint to $!_\Gamma^*$.\\\\
By \ref{zero} every arrow with codomain $0$ is iso. So the pullback of $f:\Delta\arr\Gamma$ along the projection $0\arr\Gamma$ is necessarily of the form
\[
\xymatrix{
0\ar[d]_-{id_0}\ar[r]^-{!_\Delta}&\Delta\ar[d]^-{f}\\
0\ar[r]_-{!_\Gamma}&\Gamma
}
\]
Hence $f^*\Sigma_{!_\Gamma}\bot_0 = f^*\bot_\Gamma=\bot_\Delta=\Sigma_{!_\Delta}\bot_0= id_0^*\Sigma_{!_\Delta}\bot_0$.
\end{proof}
Lemmas \ref{BC} and \ref{nonne0} have the following corollary.
\begin{prop}\label{nonne1} A primary doctrine which satisfies the hypothesis of proposition \ref{nonne0} is existential.
\end{prop}
\begin{proof} Suppose $\pp$ is such a doctrine. After \ref{BC} and \ref{nonne0} it remains to prove that $(\C,P)$ satisfy Frobenius Reciprocity. Consider a projection $\pi_\Gamma:\Gamma\times A\arr A$\\\\
Suppose $A$ is not a stable initial object and for $\phi$ in $P(\Gamma\times A)$ and $\beta$ in $P(A)$ denote by $\xi$ the formula $\phi\wedge\pi_\Gamma^*\beta$. We have  $$\Sigma_{\pi_\Gamma}(\phi\wedge\pi_\Gamma^*\beta)= <id_\Gamma,\ee_\xi>^*(\phi\wedge\pi_\Gamma^*\beta)$$
Moreover $$<id_\Gamma,\ee_\phi>^*\phi\wedge \beta = <id_\Gamma,\ee_\phi>^*(\phi\wedge\pi_\Gamma^*\beta)$$
apply lemma \ref{BC} to get $$\Sigma_{\pi_\Gamma}\phi\wedge \beta=<id_\Gamma,\ee_\phi>^*\phi\wedge \beta \le <id_\Gamma,\ee_\xi>^*(\phi\wedge\pi_\Gamma^*\beta)=\Sigma_{\pi_\Gamma}(\phi\wedge\pi_\Gamma^*\beta)$$
from which the claim follows, as the other inequality is standard.\\\\
If $A\iso 0$ the the claim is trivial since the image of the left adjoints is a bottom element.
\end{proof}
\section{Heacos and triposes}\label{sec5}
In this section we introduce the notion of Heaco which is an acronym for \textbf{H}igher order \textbf{E}lementary doctrine validating \textbf{A}C and with \textbf{CO}-comprehension. We also require that all these structures well interact together in the sense specified below, in the formal definition of heaco. Most of the section is devoted to prove that a heaco $\pp$ has enough structure to construct a tripos with base $\C$.\\\\
Suppose $(\C,P)$ is a doctrine. To every poset $P(A)$, we can associate the poset $P(A)^{op}$ where $\alpha\le \beta$ in $P(A)^{op}$ if and only if $\beta\le\alpha$ in $P(A)$. The functor $$P^o:\C^{op}\arr\textbf{Pos}$$
which maps every object $A$ of $\C$ in $P^o(A)=P(A)^{op}$ and every arrow $f:A\arr B$ of $\C$ in $P^o(f) = P(f)=f^*$ gives rise to a doctrine $(\C,P^o)$.\\\\
It is straightforward to verify that the assignment $(\C,P)\mapsto(\C,P^o)$ is nilpotent, i.e. $(\C,(P^o)^o)=(\C,P)$. Moreover $(\C,P)$ has (full) comprehension if and only if $(\C,P^o)$ has (full) co-comprehension.
\begin{lem}\label{sinistra} If $(\C,P)$ is a higher order $\Sigma$-doctrine with full co-comprehension which is also a restricted $\Sigma(\C_P^o)$-doctrine, then $(\C,P^o)$ is a tripos with full comprehension.
\end{lem}
\begin{proof} Since $\pp$ is a $\Sigma$-doctrine, $(\C,P^o)$ is clearly a $\Pi$-doctrine. For the same reason $(\C,P^o)$ is a restricted $\Pi(\C_P^o)$-doctrine, since $\C_P^o$ coincides with $\C_{P^o}$, the class of full comprehension of $(\C,P^o)$. The claim follows from \ref{redundus} and \ref{bingo}. 
\end{proof}
Recall that in an elementary doctrine $(\C,P)$ for every arrow $f:X\arr A$ of $\C$ the formula $$\mathcal{G}(f)=(f\times id_A)^*\delta_X$$ is called graph of $f$.
\begin{deff} An \textbf{eaco} is an elementary doctrine $\pp$ with full co-comprehension satisfying AC, such that: for every $f:X\arr A$ in $\C$ and every $\alpha$ in $P(A)$ it holds that
$$f^*<\ee_{\mathcal{G}(\lceil\alpha\rceil)}, id_A>^*\mathcal{G}(\lceil\alpha\rceil)=<\ee_{\mathcal{G}(\lceil f^*\alpha\rceil)}, id_X>^*\mathcal{G}(\lceil f^*\alpha\rceil)$$
A \textbf{heaco} is an eaco which is higher order.
\end{deff}
Suppose $(\C,P)$ is an eaco. Suppose that $\C$ has a stable initial object $0$.
\begin{prop}\label{checazzo2} $\PP(0)$ is a singleton.
\end{prop}
\begin{proof} Suppose $\xi$ is in $\PP(0)$ and consider its co-comprehension $$\lceil\xi\rceil:\{\xi\}^o\arr 0$$ By \ref{zero} $\lceil\xi\rceil$ is iso. By fullness of co-comprehension $\xi= \bot_0$.
\end{proof}
\begin{prop}Every eaco is a an existential doctrine.
\end{prop}
\begin{proof} This is a corollary of \ref{nonne1} and \ref{checazzo2}. 
\end{proof}
Since an eaco $(\C,P)$ is also elementary we have that for every $f:A\arr B$ the functor $f^*$ has a left adjoint $\Sigma_f$. Recall that we denoted by $\C^o_P$ the class of arrows of $\C$ of the form $\lceil\alpha\rceil$.
\begin{prop}\label{baggins} Every eaco $(\C,P)$ is a restricted $\Sigma(\C_P^o)$-doctrine.
\end{prop}
\begin{proof} Suppose $(\C,P)$ is an eaco. Suppose also that $f:X\arr A$ is an arrow in $\C$ and $\alpha$ is in $P(A)$. Consider the diagram
\[
\xymatrix{
\{f^*\alpha\}^o\ar[r]^-{q}\ar[d]_-{\lceil f^*\alpha\rceil}&\{\alpha\}^o\ar[d]^-{\lceil\alpha\rceil}\\
X\ar[r]_-{f}&A
}
\]
Since an eaco is existential and elementary, we have that the functors $\lceil\alpha\rceil^*$ has a left adjoint $\Sigma_{\lceil\alpha\rceil}$. We want to prove the restricted Beck-Chevalley condition$$f^*\Sigma_{\lceil\alpha\rceil}\lceil\alpha\rceil^*\beta = \Sigma_{\lceil f^*\alpha\rceil}q^*\lceil\alpha\rceil^*\beta$$
An eaco validates AC, then it is existential as it validates Frobenius reciprocity. Moreovoer equality predicates are substitutive. Then we have that the left hand side of the previous equality is
\begin{equation}\notag
\begin{split}
f^*\Sigma_{\lceil\alpha\rceil}\lceil\alpha\rceil^*\beta & = f^*\Sigma_{\pi_A}(\lceil\alpha\rceil\times id_A)^*\delta_A\wedge \pi_{\{\alpha\}}^*\lceil\alpha\rceil^*\beta\\
& = f^*\Sigma_{\pi_A}(\lceil\alpha\rceil\times id_A)^*\delta_A\wedge \pi_A^*\beta\\
& = f^*\beta\wedge f^*\Sigma_{\pi_A}(\lceil\alpha\rceil\times id_A)^*\delta_A
\end{split}
\end{equation}
For analogous reasons, the right hand side is 
\begin{equation}\notag
\begin{split}
\Sigma_{\lceil f^*\alpha\rceil}q^*\lceil\alpha\rceil^*\beta & = \Sigma_{\lceil f^*\alpha\rceil}\lceil f^*\alpha\rceil^*f^*\beta\\
& = f^*\beta \wedge \Sigma_{\lceil f^*\alpha\rceil}\top_{s}\\
& = f^*\beta\wedge\Sigma_{\pi_X}(\lceil f^*\alpha\rceil\times id_X)^*\delta_X
\end{split}
\end{equation}
Finally
\begin{equation}\notag
\begin{split}
f^*\Sigma_{\pi_A}(\lceil\alpha\rceil\times id_A)^*\delta_A& = f^*\Sigma_{\pi_A}\mathcal{G}(\lceil\alpha\rceil)\\ 
& = f^*<\ee_{\mathcal{G}(\lceil\alpha\rceil)}, id_A>^*\mathcal{G}(\lceil\alpha\rceil)\\ 
& = <\ee_{\mathcal{G}(\lceil f^*\alpha\rceil)}, id_X>^*\mathcal{G}(\lceil f^*\alpha\rceil)\\
& = \Sigma_{\pi_X}\mathcal{G}(\lceil f^*\alpha\rceil)\\ 
& = \Sigma_{\pi_X}(\lceil f^*\alpha\rceil\times id_X)^*\delta_X
\end{split}
\end{equation}
from which the claim. 
\end{proof}
The immediate application of \ref{sinistra} and \ref{baggins} is that every heaco $\pp$ gives rise to a tripos $(\C,P^o)$, this allows us to prove the following propositions.
\begin{prop}\label{frodo} Every heaco is a $\Pi$-doctrine.
\end{prop}
\begin{proof} Suppose $\pp$ is a heaco. By \ref{sinistra} and \ref{baggins} $(\C,P^o)$ is a tripos and therefore a $\Sigma$-doctrine. Then $(\C,(P^o)^o)=\pp$ is a $\Pi$-doctrine.
\end{proof}
Using the same argument we can infer that in every heaco, posets of the form $P(A)$ have finite joins (since $P(A)^o$ is a Heyting algebra).
\begin{prop}\label{caratterino}
Suppose $\pp$ is an heaco, the following are equivalent
\begin{itemize}
\item[i)] $\pp$ is implicational
\item[ii)] $\pp$ is a tripos
\end{itemize}
\end{prop}
\begin{proof} ii)$\Rightarrow$i) is immediate. The converse is also immediate after \ref{frodo} and \ref{redundus}.
\end{proof}

\end{document}